\documentclass{amsart}

\title[Time Average]{Time averages of polynomials}
\author{Han Peters}
\usepackage{amscd}

\newtheorem{theorem}{Theorem}
\newtheorem{lemma}[theorem]{Lemma}

\newtheorem{corollary}[theorem]{Corollary}

\newtheorem{conjecture}[theorem]{Conjecture}
\newtheorem{question}[theorem]{Question}

\theoremstyle{definition}

\theoremstyle{remark}

\newcommand{\NN}{\mathbb{N}}
\newcommand{\RR}{\mathbb{R}}
\newcommand{\CC}{\mathbb{C}}

\def\a{{\alpha}}

\def\e{{\epsilon}}

\begin{document}

\bibliographystyle{plain}

\begin{abstract}
We define and study when a polynomial mapping has a local or global time average. We conjecture that a polynomial $f$ in the complex plane has a time average near a point $z$ if and only if $z$ is eventually mapped into a Siegel-disc of $f$. We prove that the conjecture holds generically, namely for those polynomials whose iterates have the maximal number of critical values. Important steps in the proofs rely on understanding the iterated monodromy groups. We also show that a polynomial automorphism of $\CC^2$ has a global time average if and only if the map is conjugate to an elementary mapping. The definition of a time average is motivated by an attempt to understand the polynomial automorphism groups in dimensions $3$ and higher.
\end{abstract}

\maketitle

\section{Introduction}

Let $f : \CC^k \mapsto \CC^k$ be an entire function and let $z \in \CC^k$. We say that $f$ has a (weighted) time average near $z$ if there exists a neighborhood $U(z)$, a constant $c \in \CC$ and a non-zero sequence $a_0, a_1, \ldots$ such that the maps
\begin{align}\label{definition}
F_N =  \sum_{n=0}^{N} a_n f^n
\end{align}
converge uniformly on $U$ to the constant function $c$ as $N \rightarrow \infty$. If the maps $F_N$ converge uniformly on any compact subset of $\CC^k$ then we say that $f$ has a global time average.

The terminology is clear, we are taking an average of the maps $\mathrm{Id} = f^0, f, f^2 = f\circ f, \ldots$ with weights $a_0, a_1 \ldots$ and we require that the sums converge to the same constant $c$ independently of the point $w \in U(z)$. The definition is motivated by an attempt to understand the structure of polynomial automorphism groups in dimensions $3$ and higher, we will explain this in more detail at the end of the introduction.

In most of this article we will study time averages for polynomials in the complex plane. Theorem \ref{global} will show that a polynomial has a global time average if and only if it is affine. We conjecture the following:

\begin{conjecture}\label{main}
A polynomial $f$ of degree at least $2$ has a time average near $z \in \CC$ if and only if $z$ is eventually mapped into a Siegel-disc of $f$.
\end{conjecture}

One direction we can easily show: Theorem \ref{Siegel} says that if $z$ is eventually mapped into a Siegel disc then $f$ admits a time average near $z$. We will not prove the other direction in full generality, but we will see that the conjecture holds in many cases. In Theorem \ref{escaping} we will show that $f$ has no time average near $z$ if $z$ lies in the escaping set $I(f)$, that is, if $f^n(z) \rightarrow \infty$. Then we will show (Corollary \ref{solution}) that the conjecture holds if $f^n$ has exactly $n(d-1)$ critical values for every $n \in \NN$. Note that $n(d-1)$ is the largest number possible and the condition is satisfied for a generic polynomial. The proof relies on the fact that the iterated monodromy groups of such a polynomial are as large as possible.

We will assume that the reader is familiar with basic notions in complex dynamical systems. The definitions and properties that we will use can be found in standard texts on the subject, see for example the books by Beardon \cite{beardonbook} and Milnor \cite{milnorbook}.

Before we go into the details of our results and proofs we explain how our definition of time averages is motivated by a study of polynomial automorphism groups in several variables. In one variable the only invertible polynomials are the affine functions. However, in dimensions $2$ and higher there exist many invertible polynomial mappings. For example, for any polynomial $p(z_1, \ldots , z_{n-1})$ one easily sees that the mapping

\begin{align}\label{elementary}
(z_1, \ldots z_{j-1}, z_j, z_{j+1}, \ldots z_n) \mapsto (z_1, \ldots, z_{j-1}, \alpha z_j + p, z_{j+1}, \ldots z_n),
\end{align}
is invertible for any polynomial $p(z_1, \cdots z_{j-1}, z_{j+1}, \ldots z_n)$ and constant $\a \neq 0$.

Automorphisms of the form \eqref{elementary} are called {\emph{elementary}}, and finite compositions of elementary automorphisms are called {\emph{tame}}. Jung's Theorem \cite{jung} states that all polynomial automorphisms in 2 variables are tame. In other words, the polynomial automorphism group of $\CC^2$ is generated by elementary automorphisms. This result has proved to be crucial for our understanding of the dynamics of polynomial automorphisms in $\CC^2$. Jung's result was used by Friedland and Milnor \cite{fm} to show that every polynomial automorphisms is conjugate to either a single elementary map or a composition of {\emph{H\'{e}non mappings}}:

\begin{align*}
(z, w) \mapsto (\delta w + p(z) , z),
\end{align*}
where $0 \neq \delta \in \CC$ and $p(z)$ a polynomial.

The dynamical behavior of a single elementary map can easily be understood, but the dynamics of H\'{e}non mappings is highly non-trivial and has been studied extensively (see for example the work by Hubbard-Oberste Vorth \cite{huob1}, \cite{huob2} Bedford-Smillie \cite{besm1}, \cite{besm2} \cite{besm3} and Forn{\ae}ss-Sibony \cite{fosi1} \cite{fosi2}).

In dimensions $3$ and higher there is no known generalization of Jung's theorem and the dynamics of polynomial automorphisms is far less well understood. In fact, it was recently shown by Stestakov and Umirbaev \cite{shum} that there exist polynomial automorphisms in $3$ dimensions that are not tame. An example is the Nagata automorphism:
\begin{align*}
F(x, y, z) = (x + (x^2 - yx)z, y + 2 (x^2-yz)x + (x^2 - yz)^2z, z).
\end{align*}

It might greatly increase our understanding of the dynamics of polynomial automorphisms in dimensions $3$ and higher if there was a known set of generators of the polynomial automorphism groups whose dynamics is well understood. Unfortunately such a set of generators is currently not known.

There are several conjectured generator sets. One possible candidate is the set of {\emph{locally finite}} polynomial automorphisms. A polynomial endomorphism $f$ is called locally finite if $\sup_{n \in \NN} \mathrm{deg}(f^n) < \infty$. Equivalently $f$ is locally finite if there exist $d \in \NN$ and $a_1, \ldots a_d \in \CC$ such that
\begin{align}\label{polyroots}
\a_1 f + a_2 f^2 + \ldots + a_d f^d = 0.
\end{align}

Clearly all elementary automorphisms are locally finite, and the Nagata automorphism is also locally finite. However, it may be that this class is still too small.

In \cite{mape} a generalization of locally finite maps was studied by Maubach and the author. Instead of looking at polynomial endomorphisms that are roots of polynomials as in Equation \eqref{polyroots}, one could consider the maps that are roots of power series:

\begin{align} \label{powerseries}
\sum_{n=1}^{\infty} a_n f^n = 0.
\end{align}

Of course, this definition depends on the kind of convergence. In \cite{mape} point-wise convergence of each coefficient was considered. It was shown that the polynomial automorphisms that are roots of a power series do generate the automorphisms groups in any dimension but in a rather trivial way: every polynomial automorphism is the composition of an affine map with a single root of a power series. Clearly this definition of a root of a power series is too weak to obtain a deeper understanding of the dynamical behavior of polynomial automorphisms.

Point-wise convergence of each coefficient in \eqref{powerseries} is equivalent to requiring that for $\a \in \NN^k$ one has that

\begin{align}\label{derivatives}
\sum_{n=1}^{\infty} a_n D_\a(f^n)(0) = 0.
\end{align}

From an analysis point of view it is more natural to consider a different kind of convergence, namely uniform convergence on compact subsets. In this article we require that the maps $\sum_{n=0}^{N} a_n f^n$ converge uniformly to a constant function either in a neighborhood of some point $z$ or on any compact subset of $\CC^k$. This of course implies point-wise convergence to $0$ of all the derivatives at $z$ (as in \eqref{derivatives}), so having a time average near $z=0$ implies that $f$ is the root of a power series in the sense discussed in \cite{mape}. Having a time average near $0$ is in fact strictly stronger. For example, in \cite{mape} it was shown that if $f$ is a polynomial and $0 = f(0)$ is a repelling fixed point, then $f$ is a root of a power series. But if $0$ is a repelling fixed point then it lies in the Julia set, and in Corollary \ref{julia} we will see that $f$ can not have a time average near the origin.

In the next section we prove (Theorem \ref{global}) that a polynomial of degree at least $2$ cannot have a global time average. In Section (2) we also show (Theorem \ref{Siegel}) that a polynomial does have a time average near a point that is eventually mapped into a Siegel disc. In Section (3) we discuss iterated monodromy groups and show how they can be used to prove Conjecture \ref{main}. Lemma \ref{monodromy} plays a central role in this paper. In Section (3) we also show (Theorem \ref{escaping} and Corollary \ref{julia}) that a polynomial cannot have a time average near a point in the escaping set or in the Julia set.

In Section (4) we prove our main result (Corollary \ref{solution}) which says that Conjecture \ref{main} holds for a generic set of polynomials, namely those whose iterates have the largest possible number of distinct critical values. In Section (5) we show that Conjecture \ref{main} does hold for some polynomials with a very small number of critical points. In particular we show (Theorem \ref{degreetwo}) that Conjecture \ref{main} holds for polynomials of degree $2$. In the sixth and last section we consider polynomial mappings in dimensions $3$ and higher. We show (Theorem \ref{henon}) that a polynomial automorphism of $\CC^2$ has a global time average if and only if it is polynomially conjugate to either an elementary or affine mapping.

The author would like to thank Rafe Jones for useful discussions on the subject of monodromy groups, and in particular for the proof of Theorem \ref{post-critical}

\section{Global time averages and Siegel discs}

\begin{theorem}\label{global}
Let $f$ be a polynomial of degree $d \ge 2$. Then $f$ does not admit a global time average.
\end{theorem}
\begin{proof}
Suppose that we have $c \in \CC$ and $a_0, a_1, \ldots$ such that
$$
\sum_{n=0}^{\infty}a_nf^n(w) = c,
$$
for all $w \in \CC$. We need to show that $a_n = 0$ for all $n \in \NN$, so let us suppose for the purpose of contradiction that $a_{N-1} \neq 0$. Choose $p \in \CC$ generically such that there are exactly $d^{N}$ distinct pre-images $f^{N}(p)_j$, and write $S$ for the set of all such pre-images. For every $w \in \CC$ we have the following:
\begin{align}\label{eqn1}
\sum_{n=0}^{N-1} a_n f^n(w) = c - \sum_{n= N}^{\infty}a_n f^n(w).
\end{align}

The right hand side of \eqref{eqn1} is constant on the set $S$, so the left hand side is also constant on $S$. But the left hand side is a polynomial of degree at most $d^{N-1}$, and $S$ contains $d^{N}$ distinct points. Hence the left hand side of \eqref{eqn1} must be constant and $a_{N-1} = 0$ which is a contradiction.
\end{proof}

The idea of this proof will be used to show that in many cases $f$ does not have a local time average. The difficulty is that typically only a small portion of the pre-images $f^{N}(p)_j$ lie in the neighborhood $U(z)$, so we cannot automatically use the same argument. The main idea is to first show that the maps $F_N$ converge uniformly in a larger set $\tilde{U}$, so that $\tilde{U}$ contains at least a few of the pre-images. Then, we use the monodromy group of $f^n$ to include the rest of the pre-images, we will describe this method more precisely in the next section.

First we will show that a polynomial does have a time average near a point that is eventually mapped into a Siegel disc. The proof does not rely on the fact that the map is a polynomial, so we will prove the result in slightly greater generality.

\begin{theorem}\label{Siegel}
Let $f: \CC \rightarrow \CC$ be an entire function and let $V
\subset \CC$ be a Fatou component of $f$ that is eventually mapped
onto a periodic Fatou component that is either a Siegel disc or a
Herman ring. Let $U$ be a relatively compact subset of $V$. Then there exists a non-zero sequence $a_0, a_1 \ldots \in
\CC$ such that
$$
\sum_{n=0}^{N} a_n f^N(w) \rightarrow 0 \; \; \mathrm{as} \; \; N
\rightarrow \infty,
$$
uniformly for $w \in U$.
\end{theorem}
\begin{proof}
Let $N \in \NN$ be such that $f^{N}(V)$ is a periodic Fatou
component.

It follows that for every $n \ge N$ and every $\e > 0$ we can
choose $m> n$ such that $|f^m(w) - f^n(w)| < \e$ for all $w \in U$.
We now inductively construct a sequence $a_N, a_{n_1}, a_{n_2}
\ldots$, the remaining weights are chosen equal to $0$. First we
choose $n_1> N$ such that
$$
\|f^N - f^{n_1} \|_U < \frac{1}{2}.
$$

Next, we choose $n_3 > n_2 > n_1$ such that
\begin{align*}
\|f^N - f^{n_1} + \frac{1}{2} f^{n_2} - \frac{1}{2} f_{n_3}\|_U
\le \\
\|\frac{1}{2} f^N - \frac{1}{2} f^{n_1}\| + \|\frac{1}{2} f^N -
\frac{1}{2} f^{n_3}\| + \|\frac{1}{2} f^{n_1} - \frac{1}{2}
f^{n_2}\| <  \frac{1}{4}.
\end{align*}

We continue by choosing $n_9 > \cdots > n_3$ such that
$$
\|f^N - f^{n_1} + \frac{1}{2} f^{n_2} - \frac{1}{2} f_{n_3} +
\frac{1}{4} f^{n_4} - \frac{1}{4} f^{n_5} - \frac{1}{4} f^{n_6} +
\frac{1}{4} f^{n_7} + \frac{1}{4} f^{n_8} - \frac{1}{4}
f^{n_9}\|_U < \frac{1}{8}
$$

Here $f^{n_4}$ is chosen very close to $f^{n_1}$ and $f^{n_5}$ is
chosen very close to the map that $f^{n_1}$ was very close to,
namely $f^N$. Similarly $f^{n_6}$ is chosen very close to
$f^{n_2}$ and $f_{n_7}$ approximates the map that $f^{n_2}$ was
very close to, namely $f^{n_1}$. And finally $f^{n_8}$
approximates $f^{n_3}$ and $f^{n_9}$ is close to $f^N$, the map
that $f^{n_3}$ is approximating.

Next $8$ higher iterates $f^{n_{10}} \cdots f^{n_{17}}$ are selected
with weights $a_{n_{10}} = \cdots = a_{n_{17}} = \frac{1}{8}$, to make
sure that the total error is less than $\frac{1}{16}$.

We continue the construction of the sequence $\{a_{n_j}\}$. We see
that the errors after the introduction of each group of
$a_{n_j}$'s converges to $0$. Furthermore, we have that $\|a_{n_j}
f^{n_j}\|_U$ converges to zero, and the estimate on the size of
$$
f^N + \sum_{j=1}^{N} a_{n_j} f^{n_j},
$$
does not increase as $N$ is increased by $2$. This completes the
proof.
\end{proof}

So note that for a point $z \in V$ we can choose any relatively compact neighborhood $U \subset \subset V$. In fact, by constructing the sequence $a_0, a_1, \ldots$ with a little care we can make sure that the sum $\sum a_n f^n$ converges uniformly on any relatively compact subset $U$.

\section{Monodromy groups}

In the proof of Theorem \ref{global} we introduced our main argument to show the non-existence of time averages. To apply this argument for local time averages we will use the monodromy groups of the iterates $f^n$. The next Lemma shows the key idea and will be used throughout the rest of the paper.

\begin{lemma}\label{monodromy}
Let $f$ be a polynomial of degree at least $2$, let $z \in \CC$, and let $a_0, a_1, \ldots \in \CC$ be such that
$$
\sum_{n = 0}^{\infty} a_n f^n = c
$$
uniformly in a neighborhood $U(z)$. Let $N \in \NN$, choose $p \in \CC$ with $d^N$ distinct pre-images $f^{-N}(z)_j$, and write $S$ for the set of these pre-images. Further suppose that there exists a nested sequence of sets $S_0 \subset S_1 \subset \cdots \subset S_m = S$ such that $S_0 \subset U$ and for each $j = 1, \ldots m$ there exists a permutation $\sigma_j$ in the monodromy-group of $f^N$ with $S_j = \sigma_j(S_{j-1}) \cup S_{j-1}$ and $\sigma_j(S_{j-1}) \cap S_{j-1} \neq \emptyset$. Then $a_0 = a_1, \ldots = a_{N-1} = 0$.
\end{lemma}
\begin{proof}
We have that
\begin{align}\label{eqn2}
F_{N-1}(w) := \sum_{n = 0}^{N-1} a_n f^n(w) = c - \sum_{n = N}^{\infty} a_n f^n (w) \; \; \mathrm{for} \; \; w \in U.
\end{align}

Let $V$ be a small neighborhood of $p$ such that the inverse branches $\{(f^{-N})_j\}$ are well-defined on $V$. If $i, j$ are such that $(f^{-N})_i$ and $(f^{-N})_j$ both map $V$ into $U$ then it follows from Equation \eqref{eqn2} that

\begin{align}\label{eqn3}
F_{N-1}\circ (f^{-N})_i = F_{N-1}\circ (f^{-N})_j.
\end{align}

We can extend the inverse branches in a neighborhood of any closed loop that avoids the critical values of $f^N$. Any element in the monodromy group of $f^N$ can be represented by a closed loop at $p$. By following the inverse branches along a loop representing $\sigma_1$ it follows from the hypotheses on $\sigma_1$ that Equation \eqref{eqn3} holds for all $i, j \in S_1$. By applying the same argument for $\sigma_2, \sigma_3$ and so on we obtain that Equation \eqref{eqn3} holds for $i, j \in S_m = S$. But that means that $F_{N-1}$ is constant on the set $S$, a set whose order is larger than the degree of $F_{N-1}$, hence $F_{N-1}$ is constant. But then $a_{N-1} = 0$ and thus $a_{N-2} = 0$ and so on.
\end{proof}

In the case that $z$ lies in the escaping set $I(f) = \{z \in \CC \mid |f^n(z)| \rightarrow \infty\}$ we can use Lemma \ref{monodromy} without having any information about the monodromy-group of $f$. We denote by $G$ the Green's function of $f$ given by
\begin{align}
G(z) = \lim_{n \rightarrow \infty} \frac{1}{d^n} \mathrm{log}^+ |f^n(z)|.
\end{align}

The function $G$ is harmonic on the escaping set and measures how fast a point $z$ escapes to infinity. Note that if $z$ and $w$ satisfy $f^n(z) = f^n(w)$ for some integer $n$ then $G(z) = G(w)$. We will use the fact that for $z$ in a neighborhood of infinity the level curve $\{w \in \CC \mid G(w) = G(z)\}$ is connected.

\begin{lemma}\label{monodromy-escaping}
Let $f$ be a polynomial of degree at least $2$, and let $z \in
\CC$ lie in the escaping set of $f$. Let $U$ be an open
neighborhood of the connected component of $\{w \in \CC \mid G(w)
\le G(z)\}$ that contains $z$.

Suppose that $f^N$ has $d^N$ distinct pre-images $\zeta_1, \ldots \zeta_{d^N}$ of a point $p \in
\CC$, with $G(\zeta_j) = G(z)$
for all $j$. Suppose further that $N \in \NN$ is chosen large enough such
that the number of pre-images $\zeta_j \in U$ is strictly larger
than the number of connected components of $\{w \in \CC \mid G(w)
\le G(z)\}$.

Then there exists a complete cycle $\sigma \in S^{d^N}$ in the
monodromy-group of $f^N$ that maps a pre-image $\zeta_j \in U$
back into $U$.
\end{lemma}
\begin{proof}
We construct a loop $\gamma$ that defines the cycle $\sigma$ explicitly. The point $p$ necessarily lies in the escaping set. We note that we may assume that the outward external ray through $p$ does not hit any critical values of $f^N$, if it does then we can start $\gamma$ by following the level curve of $G$ slightly.

The loop $\gamma$ will consist of three pieces, $\gamma_1$, $\gamma_2$ and $\gamma_3$. We let $\gamma_1$ follow the outward external ray through $p$ far enough to a point $q$ so that all pre-images $f^{-N}(q)_j$ lie on a simple level-curve of $G$ that is connected.

Then $\gamma_2$ follows one full clockwise rotation along the level curve of $G$. It follows that the pre-images of $q$ must also rotate clockwise along the inverse branches of $\gamma_2$ until they arrive at the next pre-image of $q$.

Finally we let $\gamma_3$ follow the external ray back from $q$ to $p$.

It is clear that $\gamma_2$ induces a complete cycle on the pre-images of $q$, and therefore $\gamma$ induces a complete cycle on the pre-images of $p$. To see that at least one pre-image in $U$ must be mapped into $U$ we claim that there exist at least two pre-images of $p$ in $U$ whose inverse branches of $\gamma_1$ lead to adjacent pre-images on the connected level curve of $G$. Of course the pre-images in $U$ start out lying adjacent on their level curve. When we follow $\gamma_1$ the level curve may split every time the connected component joins with one or more of the other connected components. The claim follows from the fact that there are more pre-images of $p$ than connected components of the level curve $\{w \in \CC \mid G(w) = G(z)\}$.
\end{proof}

Note that the monodromy group is independent of the point $p$, and hence there always exists a complete cycle. The important point is that the cycle we constructed above maps a point in $U$ to a point in $U$. It is exactly this property that allows us to prove the following result.

\begin{theorem}\label{escaping}
Let $f$ be a polynomial and let $z \in I(f)$. Then $f$ does not have a time average near $z$.
\end{theorem}
\begin{proof}
Let $U$ be a neighborhood of $z$ and assume that there exist $c \in \CC$ and sequence $a_0, a_1, \ldots \in \CC$ such that
\begin{align}\label{eqn4}
\sum_{n=0}^{n=\infty} a_n f^n \rightarrow c,
\end{align}
uniformly on $U$.

Since $f^n(z)$ converges to infinity and the polynomial $f$ behaves like its highest degree term near infinity, one can easily see that for large $N \in \NN$, the set $f^N(U)$ must contain a large circle centered at the origin, and in particular must contain the whole (connected) level curve of the Green's function $G$ through  $f^N(z)$. Since the sum on the left hand side of Equation \eqref{eqn4} converges uniformly on $U$ we have see the sum
$$
\sum_{m=0}^{m=\infty} a_{m+N} f^m
$$
converges uniformly on $f^N(U)$, and by the maximal principle this sum must also converge uniformly on the area $V$ enclosed by $f^N(U)$.

It follows that the sum $\sum_{n=0}^{n=\infty} a_n f^n$ converges uniformly on $f^{-N}(V)$. Let us denote by $\tilde{U}$ the connected component of $f^{-N}(V)$ that contains $z$. Then we have that for large $N$ the set $\tilde{U}$ must contain the connected component of the level curve of $G$ through $z$. By connectivity we have that Equation \ref{eqn4} holds on $\tilde{U}$.

Now choose $N$ large enough and $p$ a point with $d^N$ pre-images of $f^N$, some of which lie on the level curve of $G$ through $z$. Let $S$ be the set of all pre-images $f^{-N}(p)_j$ and let $\sigma$ be the complete cycle acting on $S$ that we constructed in Lemma \ref{monodromy-escaping}. Define $S_0 = S \cap \tilde{U}$ and $S_{j+1} = S_j \cup \sigma(S_j)$ for $j \in \NN$. Then $S_{d^N} = S$ and Lemma \ref{monodromy} gives that $a_{N-1} = a_{N-2} = \cdots = a_1 = 0$. Since the argument holds for arbitrarily large $N \in \NN$ we see that $f$ does not have a time average at $z$.
\end{proof}

Let us recall a few well-known definitions and results regarding the dynamics of polynomials in the complex plane. We denote by $F(f)$ and $J(f)$ respectively the Fatou set and Julia set of $f$. In the case of a polynomial the Julia set is exactly equal to the boundary of the escaping set. Hence Theorem \ref{escaping} immediately implies:

\begin{corollary}\label{julia}
Let $f$ be a polynomial and let $z \in J(f)$. Then $f$ does not have a time average near $z$.
\end{corollary}

So from now on we only have to consider the case where $z$ lies in a bounded Fatou component of $f$. By Sullivan's Non-Wandering Theorem \cite{sullivan} we know that every Fatou component is eventually mapped onto a periodic Fatou component. Moreover, the periodic Fatou components are completely classified. A bounded Fatou component of a polynomial must be one of the following: (1) an attracting basin, (2) a super-attracting basin, (3) an attracting petal or (4) a Siegel disc. We have already solved the case of a Siegel disc but the other three cases remain. A clear difference between these three kinds of Fatou components and a Siegel disc is that the orbit of these Fatou components must contain a critical point. It is exactly that distinction that is relevant here, and it will allow us to prove the main conjecture \ref{main} for many polynomials.

\section{Polynomials with the maximal number of critical values}

In the previous section we saw how the iterated monodromy groups may be used to show the non-existence of a local time average. In Theorem \ref{escaping} we put no restrictions on the polynomial $f$, but made assumptions on the orbit of $z \in \CC$. In this section we will show that for a large class of polynomials Conjecture we can determine exactly what the monodromy group is and use this to prove that Conjecture \ref{main} holds.

The iterated monodromy groups of a polynomial $f$ can be thought of as all acting on a the same tree of pre-images of a point $p$, with on the first level of the tree the pre-images $f^{-1}(p)_j$, and on the second level of the tree the pre-images of those points, i.e. the pre-images $f^{-2}(p)$ and so on. The monodromy elements must respect the structure of this tree, so the action of such an element on higher levels of the tree is restricted by the action on the lower levels. We refer the reader to the book \cite{Nekrabook} for a more detailed explanation.

 Let us start by computing the iterated monodromy groups for the example $f(z) = z^2 + 1$. The polynomial $f$ has only one critical value, namely the point $1$, so the monodromy group of $f$ is generated by a single element. Let us represent this element by a closed loop $\gamma_1$ that starts at the point $0$, then moves towards the critical value $1$ along the real axis, then turns clock-wise once around $1$ and moves back along the real axis to the starting point $0$. The inverse images of $0$ are the points $i$ and $-i$. When we follow the inverse branches along the loop $\gamma_1$ we see that $i$ goes to $-i$ and $-i$ goes to $+i$. So the loop $\gamma_1$ induces a switch of the two inverse images $i$ and $-i$.

The map $f^2 = f \circ f$ has two critical values, $1$ and $2$. We introduce the loop $\gamma_2$ that moves from $0$ to $1-\e$, then makes half a clockwise turn around $1$ to $1+ \e$, follows the real axis to $2-\e$, turns clockwise around $2$ and follows the original path back to $0$. The loops $\gamma_1$ and $\gamma_2$ generate the fundamental group of $\CC \setminus \{1. 2\}$ so the monodromy elements induced by $\gamma_1$ and $\gamma_2$ generate the monodromy group of $f^2$. To understand the monodromy group we need to know how the inverse images of $\gamma_1$ and $\gamma_2$ act on the four inverse images $f^{-2}(0)_j$, so on the two inverse images $f^{-1}(i)_j$ and $f^{-1}(-i)_j$.

First note that the loop $\gamma_2$ does not wind around the critical value $1$, so an inverse image of $i$ must go to an inverse image of $i$ and likewise for $-i$. An explicit computation shows that the inverse images of $i$ switch while the inverse images of $-i$ stay fixed when following the inverse branches of $\gamma_2$.

Since $\gamma_1$ defined a switch on $i$ and $-i$ it necessarily maps the pre-images of $i$ to pre-images of $-i$ and vice-versa, but it does not have any further action on higher levels, that is, $\gamma^2$ gives the trivial monodromy element. One easily sees that $\gamma_1$ and $\gamma_2$ generate the full set of tree-automorphisms.

The following result was shown to the author by Rafe Jones.

\begin{theorem}\label{post-critical}
Let $f$ be a polynomial of degree $d \ge 2$, and assume that for every $n \in \NN$ the polynomial $f^n$ has $n(d-1)$ critical points. Then for every $N \in \NN$ the monodromy-group of $f^N$ is as large as possible, that is, equal to the entire group of tree-automorphisms.
\end{theorem}
\begin{proof}
Let us first note that it follows from the proof of Lemma \ref{monodromy-escaping} that the monodromy-group of a polynomial of degree at least $2$ must always contain a complete cycle. So in particular it must be transitive.

We choose a base point $p$ distinct from a critical value of any iterate $f^n$. The first iterate $f$ has $d-1$ critical values $v_{1,1},\ldots, v_{1,d-1}$, all of multiplicity $1$. For each critical value $v_{1,j}$ we choose a closed loop $\gamma_{1,j}$ that starts at $p$, approaches $v_{1,j}$ very closely, winds around $v_{i,j}$ once and follows the original path back to $p$. When we look at the inverse branches of the loop $\gamma_{1,j}$ we notice that two branches lead very closely to the critical point $c_{1,j}$ that is mapped to $v_{1,j}$, while all the other branches lead to the $d-2$ inverse images of $v_{1,j}$ that are not critical points. When the loop winds around $v_{1,j}$ once, the two inverse branches that lead to the critical point are switched, while all other branches just wind around the non-critical point back to themselves. So we see that each loop $\gamma_{1,j}$ induces a monodromy action that switches two inverse images and fixes all other inverse images. Hence the monodromy group of $f$ is generated by $d-1$ $2$-cycles, and by transitivity we obtain all $2$-cycles and thus the whole permutation group.

Now we will proceed by induction. We assume that the monodromy group of $f^{N-1}$ is the tree automorphism group, and we will show the same holds for $f^N$. The polynomial $f^N$ has exactly $d-1$ critical points that are not critical points of $f^{N-1}$, we denote these points by $v_{N,1}, \ldots , v_{N,d-1}$. We construct closed loops $\gamma_{N,J}$ as for $N=1$. Again, each of these loops induces a 2-cycle $\sigma_{j}$ in the monodromy groups. Since by assumption the monodromy group of $f^{N-1}$ is the full tree automorphism group we can conjugate the $\sigma_j$ with elements in the monodromy group of $f^{N-1}$ (which we can view as a subgroup of the monodromy group of $f^N$) if necessary to guarantee that all the $\sigma_j$ switches two pre-images (of $f$) of the same point $f^{N-1}(p)_l$. Again by transitivity we obtain the whole permutation group on those $d$ pre-images and hence the entire group of tree automorphisms.
\end{proof}

We will see in Corollary \ref{solution} that having the monodromy groups to equal the full set of tree automorphisms is more than enough to show that Conjecture \ref{main} holds. We will need the following lemma.

\begin{lemma}\label{enlarge}
Let $f$ be a polynomial that has a time average near $z \in \CC$. Further assume that $z$ lies in a Fatou component $V$ that is eventually mapped onto a bounded periodic Fatou-component. If $V$ is eventually a periodic attracting basin or a periodic attracting petal then the maps $F_n$ defined in Equation \eqref{definition} converge uniformly on any relatively compact $\tilde{U} \subset V$. If $V$ is eventually mapped onto a periodic super-attracting basin then the maps $F_N$ converge uniformly on some connected open subset $\tilde{U} \subset V$ that contains a point which is eventually mapped to the super-attracting periodic point.
\end{lemma}
\begin{proof}
Let us first consider the case of a super-attracting basin. Let $U$ be the neighborhood of $z$ where the maps $F_N$ converge uniformly. We may assume that $U \subset V$ so that for $N$ large the set $f^{N}(U)$ lies in a small neighborhood of the attracting periodic cycle $w_0, w_1, \ldots , w_m = w_0$. If $f^m$ has a critical point of order $k$ at $w_0$ then the map $f^m$ behaves very similarly to $x \mapsto x^k$ in a neighborhood of $w_0$. Hence for large $M$ the set $f^M(f^N(U))$ will contain a small circle around the periodic point. The sum $\sum_{l=0}^{\infty}a_{N+M+l}f^l$ converges uniformly on $f^{M+N}(U)$, hence by the maximum principle also in the area enclosed by $f^{M+N}(U)$, which we will denote by $\widehat(U)$. Hence the polynomials $F_N$ converge uniformly on the set $f^{-M-N}(\widehat{U})$, so in particular on the connected component $\tilde{U} \subset f^{-M-N}(\widehat{U})$ that contains $U$. Since $\widehat{U}$ is connected it follows that $\tilde{U}$ must contain a pre-image of the periodic point.

The other two kinds of Fatou components require more technical computations and we will only outline the proof. Without loss of generality one may assume that $V = f(V)$ is a fixed Fatou component. Using standard estimates on the rate of convergence for an attracting petal one can show that for $w \in \tilde{U} \subset \subset V$ the sum
$$
\sum_{n=0}^{\infty} a_n [f^n(w) - f^n(z)]
$$
converges uniformly. This follows from the fact that if $b_0, b_1, \ldots \in \CC$ is a summable sequence and $c_0, c_1, \ldots \in \RR^+$ is a decreasing sequence then $b_0c_0, b_1c_1, \ldots$ is summable.

In the case of an attracting basin the argument is as follows. Let us denote the attracting fixed point by $q \in V$ and write $\alpha = f^\prime(q)$. Let $G$ be the function defined by
$$
\lim_{n \rightarrow \infty} \alpha^{-n} [f^n(z) - q].
$$
Then using the same summability condition as for the attracting petal one can show that the sum
$$
\sum_{n=0}^{\infty} a_n [\frac{G(z)}{G(w)} \left(f^n(w) - q\right)  - \left(f^n(z) - q\right)],
$$
converges uniformly for $w \in \tilde{U} \subset \subset V$. When $q=0$ it follows immediately that the sums $\sum a_nf^n(w)$ converge uniformly. When $q\neq 0$ we can similarly show that the sequence $a_0, a_1, \ldots$ must be summable and we obtain the same result.
\end{proof}

We have the following corollary:

\begin{corollary}\label{solution}
Let $f$ be a polynomial of degree $d \ge 2$, and assume that for every $n \in \NN$ the polynomial $f^n$ has $n(d-1)$ critical points. Then $f$ has a time average near $z \in \CC$ if and only if $z$ is eventually mapped into a Siegel disc.
\end{corollary}
\begin{proof}
By Theorem \ref{Siegel} we know that $f$ has a time average near $z$ if $z$ is eventually mapped into a Siegel disc. We have seen in Theorem \ref{escaping} and Corollary \ref{julia} that $f$ does not have a time average near $z$ if $z$ lies in the escaping set or in the Julia set. So let us assume for the purpose of contradiction that $f$ has a time average near a point $z$ that is eventually mapped into a periodic Fatou component $V$ that is either a super-attracting basin, an attracting basin or an attracting petal. Let $N $ be a large integer such that $f^N$ maps $V$ into a periodic Fatou component and $f$ is not injective on $F^{N-1}(V)$. Then choose $p \in f^N(V)$ close enough to the limit point, i.e. the periodic point either in $f^N(V)$ or on the boundary in case of an attracting petal.

It follows from Lemma \ref{enlarge} that we may assume that $U$ contains two pre-images $f^{-N}(p)_1$ and $f^{-N}(p)_2$ that lie in different main branches in the monodromy tree, in other words
$$
f^{N-1}(f^{-N}(p)_1) \neq f^{N-1}(f^{-N}(p)_2).
$$

We know from Theorem \ref{post-critical} that the monodromy group of $f^N$ is equal to the entire group of tree automorphisms. Therefore for any other pre-image $f^{-N}(p)_j$, with $j \neq 1, 2$ there is a monodromy permutation that maps $f^{-N}(p)_1$ to $f^{-N}(p)_j$ and maps $f^{-N}(p)_2$ to either $f^{-N}(p)_1$ or $f^{-N}(p)_2$. Therefore we obtain the sets $S_0, S_1, \ldots , S_m$ in Lemma \ref{monodromy} which gives that the sequence $a_0, a_1, \ldots$ for the time average is the zero-sequence which is a contradiction.
\end{proof}

\section{Polynomials with fewer critical values}

We will now look at some examples of polynomials whose iterates have fewer critical values and whose iterated monodromy groups are not equal to the tree automorphisms groups. However, for each of these examples the monodromy group will still be large enough to show that Conjecture \ref{main} holds.

First we look at the polynomial $x^2 - 1$. Both the dynamical behavior and the iterated monodromy groups of this polynomial are well-known, see for example \cite{Nekrabook}. Note that the monodromy groups of $f^n$ for large $n\in \NN$ are much smaller than the tree automorphism groups. The orbit of the unique critical point $0$ contains only the points $0$ and $1$, and therefore the iterated monodromy groups are all generated by only two elements. All bounded Fatou components are eventually mapped to the super-attracting basin of the cycle $\{0,1\}$. Let $V_0$ be the Fatou component that contains the point $-1$ and let $z \in \CC$ lie in a bounded Fatou component $V$. We assume for the purpose of contradiction that $f$ admits a time average near $z$, the argument that we will use to get a contradiction is similar to that in the proof of Theorem \ref{escaping}.

Let $m$ be the smallest positive integer such that $f^m(V) = V_0$ and let $V_1, V_2, \ldots , V_l$ be all Fatou components that are mapped onto $V_0$ by $f^m$. If $p \in V_0$ then all pre-images $f^{-m}(p)_j$ must lie in one of the sets $V_j$.

The Fatou component that contains $0$ is the only Fatou component that maps directly to $V_0$, and $f$ is $2-1$ there. Hence $V$ contains exactly two pre-images $f^{-m}(p)_j$. By choosing the point $p$ close enough to $-1$ and by using Lemma \ref{enlarge} to enlarge the set $U$ if necessary we may assume that $U$ contains exactly two pre-images.

We define a graph representing the Filled-in Julia set of a polynomial as follows. We assign a vertex to each bounded Fatou component and to each point in the Julia set. For each boundary point of a bounded Fatou component we draw an edge between the vertex representing the Fatou component and the vertex representing the boundary point. If the Filled-in Julia set is connected and has interior then we obtain a connected graph, and since the Julia set is the boundary of the escaping set we see that there is a unique simple path between any two bounded Fatou components.

Recall that the Fatou components $V_2, V_3, \ldots , V_l$ are exactly those that get mapped to $V_0$ by $f^m$, where $V_0$ is the Fatou component that contains the critical value of $f$ and $m$ is the minimal positive integer with $f^m(V) = V_0$. It follows from the fact that $f$ has only one critical point that the simple paths from the $V$-vertex to the vertices corresponding to the Fatou components $V_2, V_3, \ldots , V_l$ all start with the same edge.

We now define a closed loop starting at $p$ that moves through the Julia set into the escaping set, then turns once clockwise around the filled Julia set on a level curve of the Green's function, and then back to $p$ along the same path. We claim that this loop induces a monodromy element that is a complete cycle and maps one of the pre-images in $V$ to the other, which gives a contradiction just as in the proof of Theorem \ref{escaping}.

The fact that this loop induces a complete cycle follows immediately from the fact that the Julia set is connected. To see that one of the pre-images is mapped to the other we must argue that the paths of these two pre-images lie adjacent on the level-curve of the Green's function. This follows from the fact that the Fatou component is extremal in the set $V_1, \ldots v_l$ in the sense that all simple paths from $V$ to $v_j \neq V$ start with the same edge.

In fact, this argument holds for any polynomial of degree $2$.

\begin{theorem}\label{degreetwo}
Let $f$ be a polynomial of degree $2$. Then $f$ has a time average near $z$ if and only if $z$ is eventually mapped into a Siegel disc.
\end{theorem}
\begin{proof}
There is exactly one critical point. If the orbit of the critical point is not (pre-) periodic then the conclusion follows from Corollary \ref{solution}. If $z$ lies in a Siegel disc or in the escaping set or in the Julia set then the result follows from Theorems \ref{Siegel} and \ref{escaping} and Corollary \ref{julia}, so we may assume that $z$ is eventually mapped into a periodic Fatou component that is either a (super-) attracting basin or an attracting petal. But the orbit of such a component must contain the critical point. We let $V_0$ be the Fatou component that contains the critical value of $f$ and we choose $p$ very close to the critical value. Then we can follow exactly the same argument as for $z^2 - 1$ to see that $f$ does not have a time average near $z$.
\end{proof}

Essential in the above proof is that there is only one critical point.

The extremality of $V$ in the graph discussed above does not necessarily hold for polynomials of higher degree. Consider for example for the polynomial $f(z) = 2 z^2 (z-2)^2$. The critical points of $f$ are $0$, $1$ and $2$, the point $1$ is mapped to $2$ and $2$ is mapped to the super-attracting fixed point $0$. There are three Fatou components that are mapped to the Fatou component containing $2$, namely those that contain the points $1-\sqrt{2}$, $1$ and $1+\sqrt{2}$. The simple path in the graph between the vertices representing the Fatou components containing $1-\sqrt{2}$ and $1+\sqrt{2}$ passes through the vertex representing the Fatou component containing $1$.

This means that the argument in the proof of Theorem \ref{degreetwo} does not work for $f(z) = 2z^2(z-2)^2$, but one can explicitly compute the iterated monodromy groups of $f$ to show that Conjecture \ref{main} does hold for this polynomial. There are still enough permutations in the iterated monodromy groups to construct the sets $S_0, S_1, \ldots S_m$ in Lemma \ref{monodromy}. However, there are polynomials for which even this construction is not possible.

Consider for example the polynomial $f(z) = z^4 - 2z^2$. Notice that $f$ is the second iterate of $z^2-1$, so the iterated monodromy groups are still generated by the loops around the critical values $-1$ and $0$. Let $V$ be the Fatou component containing $\sqrt(2)$, and take a generic point $p$ in the Fatou component containing $0$. Now let $N \ge 2$ and let $S$ be the set of the $4^N$ distinct pre-images of $p$. Let $S_0 \subset S$ be the subset of those pre-images that lie in $V$. Let $\widehat{S}$ be the largest set that one can obtain using the construction in Lemma \ref{monodromy}, in other words, $\widehat{S}$ is the smallest subset of $S$ that contains $S_0$ and such that $\sigma(\widehat{S}) \cap \widehat{S}$ is either $\emptyset$ or $\widehat{S}$ for any permutation $\sigma$ in the monodromy group of $f^N$. By computing the iterated monodromy groupa of $f$ explicitly one sees that $\widehat{S}$ is not equal to $S$ but contains exactly half the elements of $S$. In fact, $\widehat{S}$ contains two of the four major branches of the tree of pre-images. However, since $\frac{1}{2} 4^N > 4^{N-1}$ one can still follow exactly the same argument as in Lemma \ref{monodromy} to see that $f$ cannot have a time average near any point in $V$.

It seems that to solve Conjecture \ref{main} using iterated monodromy groups one would need to give a positive answer to the following question:

\begin{question}
Let $V$ be a Fatou component for a polynomial $f$ of degree $d$. For $N\in \NN$ and a generic point $p \in \CC$ such that $f^N$ has $d^N$ distinct pre-images we define $\widehat{S}$ as the smallest subset of $S = \{z \mid f^N(z) = p\}$ that contains $S_0 = \{ z\in V \mid f^N(z) = p\}$ and has the property that $\sigma(\widehat{S})$ is either $\emptyset$ or $\widehat{S}$ for any permutation $\sigma$ in the monodromy group of $f^N$. Can we always find $N$ and $p$ such that the number of elements in $\widehat{S}$ is strictly larger than $d^{N-1}$?
\end{question}

We know that the answer is affirmative for generic polynomials, polynomials of degree $2$ as well as for many other polynomials.

\section{Polynomial mappings in $\CC^k$}

Some of the proofs for results from the previous sections also work in higher dimensions.  For example, if $z$ lies in a Siegel domain of $f$ (a domain where the action of $f$ is holomorphically conjugate to a rotation) then $f$ has a local time average near $z$. Also, the proof of Theorem \ref{global} immediately gives the following result:

\begin{theorem}
Let $f : \CC^k \mapsto \CC^k$ be a polynomial mapping of algebraic degree $d \ge 2$ and assume that the topological degree equals $d^k$. Then $f$ does not admit a global time average.
\end{theorem}

Of course, in higher dimensions there are many polynomial mappings for which the topological degree is strictly less than $d^k$, and such polynomial mappings may well have a global time average. As we noted in the introduction, an elementary mapping is locally finite so certainly has a global time average. In fact, for invertible polynomial mappings in $\CC^2$ we see that locally finite mappings are the only maps with global time averages:

\begin{theorem}\label{henon}
A polynomial automorphism $f: \CC^2 \rightarrow \CC^2$ admits a global time average if and only if $f$ is polynomially conjugate to an elementary or affine mapping. In particular if and only if $f$ is locally finite.
\end{theorem}
It was already shown by Furter and Maubach \cite{fuma} that a locally finite polynomial automorphism of $\CC^2$ is conjugate to an elementary or affine mapping.
\begin{proof}
Recall that a polynomial automorphism of $\CC^2$ is polynomially conjugate to either a single elementary map or affine map or to a finite composition of H{\'e}non mappings. If $f$ is conjugate to an affine or elementary mapping then the degrees of the maps $f^n$ are bounded and $f$ is locally finite, hence admits a global time average.

If $f$ is conjugate to a composition of H{\'e}non mappings then there must be a point $z$ whose forward orbit escapes to the line at infinity exponentially fast and also a point $w$ whose backward orbit escapes to infinity exponentially fast but whose forward orbit is bounded. If the sum $\sum a_n f^n$ converges for every element in $\CC^2$ then it must in particular converge for every point $f^m(z)$. It follows that there is some constant $C>0$ such that $|a_n| < \frac{C}{2^n}$. But by evaluating $\sum a_n f^n = c$ at points $f^{-m}(w)$ we see that $a_0 = 0$, and then also $a_1=0$ and so on. Hence $f$ does not admit a time average.
\end{proof}

Note that the existence of two orbits, one going exponentially fast to infinity in forward time, the other going to infinity exponentially fast in backward time yet is bounded in forward time, guarantees that there is no non-zero sum of the form $\sum a_n f^n$ that converges globally. This idea holds in dimensions $3$ and higher as well and would allow one to prove that polynomial automorphisms in dimensions $3$ and higher whose behavior is similar to that of H{\'e}non mappings do not have global time averages.

Theorem \ref{henon} raises the question whether there exist polynomial mappings (or even automorphisms) that admit a global time average but are not locally finite. The author does not know the answer.

\bibliography{biblio}

\begin{thebibliography}{10}

\bibitem{beardonbook}
Alan~F. Beardon.
\newblock {\em Iteration of rational functions}, volume 132 of {\em Graduate
  Texts in Mathematics}.
\newblock Springer-Verlag, New York, 1991.
\newblock Complex analytic dynamical systems.

\bibitem{besm1}
Eric Bedford and John Smillie.
\newblock Polynomial diffeomorphisms of {${\bf C}\sp 2$}: currents, equilibrium
  measure and hyperbolicity.
\newblock {\em Invent. Math.}, 103(1):69--99, 1991.

\bibitem{besm2}
Eric Bedford and John Smillie.
\newblock Polynomial diffeomorphisms of {${\bf C}\sp 2$}. {II}. {S}table
  manifolds and recurrence.
\newblock {\em J. Amer. Math. Soc.}, 4(4):657--679, 1991.

\bibitem{besm3}
Eric Bedford and John Smillie.
\newblock Polynomial diffeomorphisms of {$\bold C\sp 2$}. {III}. {E}rgodicity,
  exponents and entropy of the equilibrium measure.
\newblock {\em Math. Ann.}, 294(3):395--420, 1992.

\bibitem{fosi1}
John~Erik Forn{\ae}ss and Nessim Sibony.
\newblock Complex {H}\'enon mappings in {${\bf C}\sp 2$} and
  {F}atou-{B}ieberbach domains.
\newblock {\em Duke Math. J.}, 65(2):345--380, 1992.

\bibitem{fosi2}
John~Erik Forn{\ae}ss and Nessim Sibony.
\newblock Complex dynamics in higher dimension. {I}.
\newblock {\em Ast\'erisque}, (222):5, 201--231, 1994.
\newblock Complex analytic methods in dynamical systems (Rio de Janeiro, 1992).

\bibitem{fm}
Shmuel Friedland and John Milnor.
\newblock Dynamical properties of plane polynomial automorphisms.
\newblock {\em Ergodic Theory Dynam. Systems}, 9(1):67--99, 1989.

\bibitem{fuma}
Jean-Philippe Furter and Stefan Maubach.
\newblock Locally finite polynomial endomorphisms.
\newblock {\em J. Pure Appl. Algebra}, 211(2):445--458, 2007.

\bibitem{huob1}
John~H. Hubbard and Ralph~W. Oberste-Vorth.
\newblock H\'enon mappings in the complex domain. {I}. {T}he global topology of
  dynamical space.
\newblock {\em Inst. Hautes \'Etudes Sci. Publ. Math.}, (79):5--46, 1994.

\bibitem{huob2}
John~H. Hubbard and Ralph~W. Oberste-Vorth.
\newblock H\'enon mappings in the complex domain. {II}. {P}rojective and
  inductive limits of polynomials.
\newblock In {\em Real and complex dynamical systems (Hiller\o d, 1993)},
  volume 464 of {\em NATO Adv. Sci. Inst. Ser. C Math. Phys. Sci.}, pages
  89--132. Kluwer Acad. Publ., Dordrecht, 1995.

\bibitem{jung}
Heinrich W.~E. Jung.
\newblock \"{U}ber ganze birationale {T}ransformationen der {E}bene.
\newblock {\em J. Reine Angew. Math.}, 184:161--174, 1942.

\bibitem{mape}
Stefan Maubach and Han Peters.
\newblock Maps that are roots of power series.
\newblock {\em To appear in Math. Z.}, 2007.

\bibitem{milnorbook}
John Milnor.
\newblock {\em Dynamics in one complex variable}.
\newblock Friedr. Vieweg \& Sohn, Braunschweig, 1999.
\newblock Introductory lectures.

\bibitem{Nekrabook}
Volodymyr Nekrashevych.
\newblock {\em Self-similar groups}, volume 117 of {\em Mathematical Surveys
  and Monographs}.
\newblock American Mathematical Society, Providence, RI, 2005.

\bibitem{shum}
Ivan~P. Shestakov and Ualbai~U. Umirbaev.
\newblock The tame and the wild automorphisms of polynomial rings in three
  variables.
\newblock {\em J. Amer. Math. Soc.}, 17(1):197--227 (electronic), 2004.

\bibitem{sullivan}
Dennis Sullivan.
\newblock Quasiconformal homeomorphisms and dynamics. {I}. {S}olution of the
  {F}atou-{J}ulia problem on wandering domains.
\newblock {\em Ann. of Math. (2)}, 122(3):401--418, 1985.

\end{thebibliography}

\center{Han Peters\\
\small University of Wisconsin\\
\small 480 Lincoln Drive, United States\\
\small peters@math.wisc.edu}
\end{document}